\documentclass[12pt]{amsproc}
  \usepackage{latexsym} 
  \usepackage[all]{xy}
  \usepackage{amsfonts} 
  \usepackage{amsthm} 
  \usepackage{amsmath} 
  \usepackage{amssymb}
  \usepackage{pifont}  
  \usepackage{enumerate}
  \xyoption{2cell}

 \usepackage{tikz}
\usetikzlibrary{arrows}

  \def\su#1{{\sp{[#1]}}}

  \def\<{{\langle}} 
  \def\>{{\rangle}} 
   
  \def\la{{\triangleright}}

  \def\cross{{>\!\!\!\triangleleft\;}} 
  \def\note#1{{}} 
 \def\can{{\rm \textsf{can}}} 

 \def\tw{{\rm \textsf{flip}}} 

  \def\note#1{} 

    \def\cL{{\mathcal L}}

  \def\cA{{\mathcal A}} 
  \def\cB{{\mathcal B}} 
  \def\cD{{\mathcal D}}

   \def\cH{{\mathcal H}}   
\def\cK{{\mathcal K}}  
  \def\cO{{\mathcal O}}
   
     \def\cT{{\mathcal T}}   
  \def\cV{{\mathfrak V}}

  \def\rhom#1#2#3{{{\rm Hom}\sb{#1}(#2,#3)}}

  \def\rend#1#2{{{\rm End}\sb{#1}(#2)}}

  \def\beq{\begin{equation}} 
  \def\eeq{\end{equation}}

  \def\id{\mathrm{id}}

  \def\ot{{\otimes}}

  \def\Endd{\mbox{\rm End}\,}

 \def\ch{{\rm \xi}}

 \def\WP{\mathbb{WP}}
 \def\pr{\mathrm{pr}}

  \newcounter{zlist}

  \newcounter{blist}

  \newcounter{rlist}

\def\stac#1{\raise-.2cm\hbox{$\stackrel{\displaystyle\otimes}{\scriptscriptstyle{#1}}$}}

\def\cten#1{\raise-.2cm\hbox{$\stackrel{\displaystyle\widehat{\otimes}}
{\scriptscriptstyle{#1}}$}}

  \def\Label#1{\label{#1}\ifmmode\llap{[#1] }\else 
  \marginpar{\smash{\hbox{\tiny [#1]}}}\fi} 
  \def\Label{\label}

  \newtheorem{proposition}{Proposition}[section]
  \newtheorem{lemma}[proposition]{Lemma} 
   
  \newtheorem{theorem}[proposition]{Theorem} 

  \theoremstyle{definition} 
  \newtheorem{definition}[proposition]{Definition}

  \theoremstyle{remark} 
  \newtheorem{remark}[proposition]{Remark}

  \newcounter{c} 
   
  \newcommand{\etyk}[1]{\vspace{-7.4mm}$$\begin{equation}\Label{#1} 
  \addtocounter{c}{1}} 
  \renewcommand{\]}{\ifnum \value{c}=1 $$\else \end{equation}\fi} 
\newcommand{\sign}{{\rm sign}}

\def\ot{\otimes}

\def\CC{{\mathbb C}}

\def\NN{{\mathbb N}}

\def\PP{{\mathbb P}}

\def\RR{{\mathbb R}}

\def\ZZ{{\mathbb Z}}

\newcommand{\Bb}{\mathcal{B}}
\newcommand{\Cc}{\mathcal{C}}

\newcommand{\Jj}{\mathcal{J}}

\newcommand{\Tt}{\mathcal{T}}

\def\*C{{}^*\hspace*{-1pt}{\Cc}}

\def\text#1{{\rm {\rm #1}}}

 \def\1{\mathbf{1}}

      \def\tr{\mathrm{Tr}\ }

\begin{document}

\title{Weighted circle actions on the Heegaard quantum sphere}

\author{Tomasz Brzezi\'nski}
 \address{ Department of Mathematics, Swansea University, 
  Singleton Park, \newline\indent  Swansea SA2 8PP, U.K.} 
  \email{T.Brzezinski@swansea.ac.uk}   
\author{Simon A.\ Fairfax}
\email{201102@swansea.ac.uk} 
 
\begin{abstract}
Weighted circle actions on the quantum Heeqaard 3-sphere are considered. The fixed point algebras, termed quantum  weighted Heegaard spheres, and their representations are classified and described on algebraic and topological levels. On the algebraic side,   coordinate algebras of quantum  weighted Heegaard spheres are interpreted as generalised Weyl algebras, quantum principal circle bundles and Fredholm modules over them are constructed, and the associated line bundles are shown to be non-trivial by an explicit calculation  of their Chern numbers. On the topological side, the $C^*$-algebras of continuous functions on quantum  weighted Heegaard spheres are described and their $K$-groups are calculated. 
\end{abstract}
\maketitle

\section{Introduction}
This article may be considered as a sequel to \cite{BrzFai:tea} in so much as, on the operator algebraic level, it develops the theory of quantum teardrops which include the generic Podle\'s two-spheres \cite{Pod:sph}. On the algebraic level, however, though it employs  techniques similar to those in \cite{BrzFai:tea} it deals with integer gradings on an algebra significantly different from that of the coordiate algebra of the quantum $SU(2)$-group.

The coordinate algebra of the Heegaard quantum 3-sphere $\cO(S_{pq \theta}^3)$ \cite{CalMat:con}, \cite{BauHaj:Hee}  is defined for parametres $0 < p,q, \theta<1$, with $\theta$ irrational, as the complex $*$-algebra  generated by  $a$ and $b$ satisfying the relations,
\begin{subequations}
\begin{equation}
ab=e^{2 \pi i \theta} ba, \qquad ab^*=e^{-2 \pi i \theta}b^*a, \label{HQS rel 1} \\
\end{equation}
\begin{equation}
a^*a-paa^*=1-p, \qquad b^*b-qbb^*=1-q, \qquad (1-aa^*)(1-bb^*)=0. \label{HQS rel 2}
\end{equation}
\end{subequations}
$\cO(S^3_{pq \theta})$ contains two copies of the quantum disc with parametres $p$ and $q$ as $*$-subalgebras and can be interpreted as obtained by glueing of two quantum solid tori \cite{BauHaj:Hee}. To describe algebraic structure of  $\cO(S^3_{pq \theta})$ it is convenient to define  (self-adjoint) $A:=1-aa^*$, $B:=1-bb^*$. In terms of these elements the relations \eqref{HQS rel 2} can be recast as:
\begin{equation}\label{HeegAB}
AB=BA=0, \quad Aa=paA, \quad  Ba=aB, \quad Ab=bA, \quad Bb=qbB.
\end{equation}
The linear basis for $\cO(S^3_{pq \theta})$  consists of all $A^ka^lb^m$, $A^ka^{*l}b^m$, $B^ka^lb^m$, $B^ka^{*l}b^m$ and their $*$-conjugates, where $k,l,m \in \NN$; see \cite{BauHaj:Hee}.

Standardly, $\cO(S^3_{pq \theta})$ is considered as a $\ZZ$-graded algebra (compatible with the $*$-structure in the sense that the $*$-operation changes the grade to its negative) in two different ways. First $a$ and $b$ are given an equal grade, say, 1. The degree zero subalgebra is generated by polynomials $A$, $B$ and $ab^*$ and is known as the coordinate algebra of the {\em mirror quantum sphere} \cite{HajMat:ind}. Second, $a$ and $b$ are given opposite grades, say 1 for $a$ and $-1$ for $b$. The degree zero subalgebra, generated by $A$, $B$ and $ab$, was introduced in \cite{CalMat:cov}, where it was shown that its $C^*$-completion is isomorphic to the algebra of continuous functions on the generic (or  non-standard) Podle\'s quantum sphere \cite{Pod:sph}. 

We prefer to interpret $\ZZ$-gradings geometrically as algebraic coactions of the coordinate Hopf algebra $\cO(U(1))$  of the circle group. $\cO(U(1))$ can be identified with the $*$-algebra $\CC[u,u^*]$ of polynomials in variables $u$ and $u^*$ satisying $uu^*=u^*u=1$. The Hopf algebra structure is given by
$$
\Delta(u)=u \otimes u, \qquad \epsilon(u)=1, \qquad S(u)=u^*.
$$
The two $\ZZ$-gradings described above correspond to coactions
$
a\mapsto a \otimes u$, $b\mapsto b \otimes u$, and  $
a\mapsto a \otimes u$, $b\mapsto b \otimes u^*$, respectively.  
 It has been shown in \cite{HajMat:loc}, \cite{HajKra:pie} that these coactions are principal, i.e.\ they make $\cO(S^3_{pq \theta})$ into a principal $\cO(U(1))$-comodule algebra \cite{BrzHaj:Che}.

The paper is organised as follows. In Section~\ref{sec.def} we equip the coordinate algebra of the Heegaard 3-sphere $\cO(S^3_{pq \theta})$  with $\ZZ$-gradings determined for a pair of coprime integers $k,l$, by setting $\deg(a) =k$, $\deg(b) =l$, or, equivalently, with the weighted coaction of $\cO(U(1))$, and study the zero degree algebras $\cO(S_{pq}(k,l))$. These split into two cases, one in which both $k$ and $l$ are positive, and one in which $k$ is positive and $l$ is negative. We use the notation $\cO(S_{pq}(k,l^\pm))$ to distinguish these cases. We list bounded irreducible $*$-representations of these algebras
and identify $\cO(S_{pq}(k,l^\pm))$ as the generalised Weyl algebras. In Section~\ref{sec.princ} subalgebras $\cA_{k,l}$ of $\cO(S^3_{pq \theta})$  which admit  principal $\cO(U(1))$-coactions that fix $\cO(S_{pq}(k,l^\pm))$ are identified. In the case $k=|l| =1$ these coincide with $\cO(S^3_{pq \theta})$ and in general they can be interpreted as quantum lens spaces; see \cite{HonSzy:len}, \cite{Zie:loc} and \cite{HajRen:len} for different constructions of lens spaces. Furthermore it is shown that these principal comodule algebras are piecewise trivial but globally non-trivial, and strong connections on them are constructed. They are also given interpretation as endomorphism rings of objects in a specific category.  Section~\ref{sec.Fred} deals with noncommutative geometric aspects of $\cO(S_{pq}(k,l^\pm))$. More concretely, we construct Fredholm modules over $\cO(S_{pq}(k,l^\pm))$ and calculate Chern numbers of line bundles associated to principal comodule algebras constructed in Section~\ref{sec.princ}. Finally, in Section~\ref{sec.cont} we study algebras of continuous functions on quantum weighted Heegaard spheres,  identify them with pullbacks of Toeplitz algebras and calculate their $K$-groups.

\section{The coordinate algebras of quantum weighted Heegaard spheres} \label{sec.def}
In this section we gather algebraic properties of quantum weighted Heegaard spheres.

\subsection{The definition of quantum weighted Heegaard spheres}  A weighted circle coaction on $\cO(S^3_{p q \theta})$ consistent with the algebraic relations $(\ref{HQS rel 1})$ and $(\ref{HQS rel 2})$ is defined for $k,l$ coprime integers, by
$$
\phi_{k,l}: \cO(S^3_{p q \theta}) \rightarrow \cO(S^3_{p q \theta}) \otimes \cO(U(1)),
$$
\begin{equation}
a \mapsto a \otimes u^k, \qquad b \mapsto b \otimes u^l,
\end{equation}
and extend to the whole of $\cO(S^3_{p q \theta})$ so that $\cO(S^3_{p q \theta})$ is a right $\cO(U(1))$-comodule algebra. The fixed point subalgebras of $\cO(S^3_{p q \theta})$ with respect to coactions  $\phi_{k,l}$ are called the coordinate algebras of {\em quantum weighted Heegaard spheres}. Equipping $\cO(S^3_{p q \theta})$ with coaction $\phi_{k,l}$ is equivalent to making it into a $\ZZ$-graded algebra with grading determined by $\deg(a) =k$, $\deg(b) =l$. At this stage we need to consider the possible signs for weights $k$ and $l$. It turns out that the fixed point subalgebra splits into two cases depending on the signs of the weights $k$ and $l$. Each case can be described by firstly putting $k>0$ and $l>0$ and secondly $k>0$ and $l<0$. We write these spaces as,
$$
\cO(S_{p q}^2(k,l^{\pm})):=\cO(S_{p q \theta}^3)^{co\cO(U(1))}=\{ x \in \cO(S_{pq \theta}^3): \phi_{k,l}(x)=x \otimes 1 \},
$$
where the $+$ sign indicates positive values for $l$ and the negative sign corresponds to negative values of $l$. Before we describe these algebras in detail, we first need some tools for calculation purposes. The following lemma  can be proven by induction.

\begin{lemma} \label{rels} For all $m,n\in \NN$,
\begin{subequations}
\begin{equation} \label{powers1}
a^{*m}a^{n}=
\begin{cases}
a^{*m-n}\prod_{i=1}^{n}(1-p^iA), & \mbox{$m \geq n$} \\
\prod_{i=1}^{m}(1-p^iA)a^{n-m}, & \mbox{$m \leq n$}.
\end{cases}
\end{equation}
\begin{equation} \label{powers2}
a^ma^{*n}=
\begin{cases}
a^{m-n}\prod_{i=1}^{n}(1-p^{-i+1}A), & \mbox{$m \geq n$} \\
\prod_{i=1}^{m}(1-p^{-i+1}A)a^{*n-m}, & \mbox{$m \leq n$},
\end{cases}
\end{equation}
and similarly with $a$ replaced by $b$ (and hence $A$ by $B$) and $p$ by $q$.
\end{subequations}
\end{lemma}
With these at hand we can describe the coordinate algebras of quantum weighted Heegaard spheres by generators and relations.

\begin{theorem}\label{thm.structure} The algebra $\cO(S^2_{pq}(k,l^+))$ is the subalgebra of $\cO(S_{pq \theta}^3)$ generated by $A=1-aa^*$, $B=1-bb^*$ and $C_+=a^l b^{*k}$ satifying the relations
\begin{equation} \label{X rel 1}
A^*=A, \quad B^*=B, \quad AB=BA=0, \quad AC_+=p^l C_+A, \quad BC_+=q^{-k} C_+B,
\end{equation}
\begin{subequations}
\begin{equation} \label{X rel 2}
C_+^*C_+=\prod_{i=1}^{l} \prod_{j=1}^{k}(1-p^iA)(1-q^{j-k}B), \qquad C_+ C_+^*=\prod_{i=1}^{l} \prod_{j=1}^{k} (1-p^{i-l}A)(1-q^{j}B).
\end{equation}
Alternatively, since $AB=0$, we can express \eqref{X rel 2} as
\begin{equation}\label{X rel 2'}
C_+^*C_+=\prod_{i=1}^{l} (1-p^iA)+\prod_{j=1}^{k}(1-q^{j-k}B)-1, \quad C_+ C_+^*=\prod_{i=1}^{l}  (1-p^{i-l}A)+\prod_{j=1}^{k}(1-q^{j}B)-1.
\end{equation}
\end{subequations}

$\cO(S^2_{pq}(k,l^-))$ is the subalgebra of $\cO(S_{pq \theta}^3)$  generated by $A=1-aa^*$, $B=1-bb^*$ and $C_-=a^{*\lvert l \rvert} b^{*k}$, satifying the relations
\begin{equation} \label{X rel 3}
A^*=A, \quad B^*=B, \quad AB=BA=0, \quad AC_-=p^l C_-A, \quad BC_-=q^{-k} C_-B,
\end{equation}
\begin{subequations}
\begin{equation} \label{X rel 4}
C_-^*C_-=\prod_{i=1}^{\rvert l \lvert} \prod_{j=1}^{k}(1-p^{i+l}A)(1-q^{j-k}B), \qquad C_- C_-^*=\prod_{i=1}^{\rvert l \lvert} \prod_{j=1}^{k} (1-p^{i}A)(1-q^{j}B).
\end{equation}
Alternatively, since $AB=0$, we can express \eqref{X rel 4} as
\begin{equation} \label{X rel 4'}
C_-^*C_-=\prod_{i=1}^{\rvert l \lvert} (1-p^{i+l}A)+ \prod_{j=1}^{k}(1-q^{j-k}B)-1, \quad C_- C_-^*=\prod_{i=1}^{\rvert l \lvert}(1-p^{i}A) +\prod_{j=1}^{k} (1-q^{j}B)-1. 
\end{equation}
\end{subequations}
\end{theorem}

\begin{proof}
The first stage is to identify the fixed point subalgebra with respect to the given coaction. A basis for $\cO(S_{p q \theta}^3)$ is given by $A^ka^{\#\mu}b^{\#\nu}$ for $k \geq 0,\mu, \nu \in \ZZ$ and $B^ka^{\#\mu}b^{\#\nu}$ for $k > 0,\mu, \nu \in \ZZ$, where 
$$
a^{\#\mu}:=\begin{cases}
a^\mu , & \mu\geq 0, \\
a^{* |\mu|},  & \mu<0,
\end{cases}$$
and similarly for the powers of $b$; see \cite{BauHaj:Hee}. First note that powers of $A$ and $B$ are automatically fixed by the weighted coaction $\phi_{k,l}$. 
Next,
$$
\phi_{k,l}(a^{\mu}b^{\nu})=a^{\mu} b^{\nu} \otimes u^{k \mu+l \nu} =  a^{\mu} b^{\nu} \otimes 1 \Longrightarrow k \mu+l \nu=0 \ \ \mbox{(the coinvariance condition)}.
$$
This means that basis elements of the form $a^{\mu}b^{\nu}$ are fixed under the coaction provided $k \mu=-l \nu$, which in turn means that $k \vert (-l \nu)$. On the other hand $k$ and $l$ are coprime so in fact $k \vert (-\nu)$ or $\alpha k=-\nu$ for some $\alpha \in \ZZ$. Substituting this back into the coinvariance condition gives
$
k \mu=l(\alpha k)$, i.e.\  $\mu=l \alpha$ and $\nu=-k \alpha.$
So,
$$
a^{\mu}b^{\nu}=a^{l \alpha}b^{-k \alpha}=(a^l)^{\alpha}(b^{*k})^{\alpha} \sim (a^l b^{*k})^{\alpha},
$$
concluding that $a^l b^{*k}$ is a generator from the set of coinvariant elements. This gives the full description of $\cO(S^2_{pq}(k,l^+))$ as the subalgebra of  $\cO(S^3_{p q \theta})$ generated by 
$A$, $B$ and $C_+=a^l b^{*k}$. Similarly, when $l$ is negative $\cO(S^2_{pq}(k,l^-))$ is generated by 
$A$, $B$ and $C_-=a^{*-l} b^{*k}$. 

Next we determine the relations between the generators for both algebras $\cO(S^2_{pq }(k,l^\pm))$, considering the positive case first. Equations~\eqref{HeegAB} immediately imply that $A^*=A$, $B^*=B$,  $AB=0$,  $
AC_+=p^lC_+A$, and $BC_+=q^{-k}C_+B$.
By Lemma~$\ref{rels}$,
\begin{equation*}
C_+C_+^*=(a^lb^{*k})(b^{k}a^{*l})=(a^la^{*l})(b^{*k}b^k)=\prod_{i=1}^{l}\prod_{j=1}^{k}(1-p^{i-l}A)(1-q^{j}B),
\end{equation*}
and 
\begin{equation*}
C_+^*C_+=(a^lb^{*k})^*(a^{l}b^{*k})=(a^{*l}a^{l})(b^{k}b^{*k}) =\prod_{i=1}^{l}\prod_{j=1}^{k}(1-p^iA)(1-q^{j-k}B).
\end{equation*}
The relations for the negative case are proven by similar arguments.
\end{proof}

\begin{remark}
The algebras $\cO(S^2_{pq }(k,l^\pm))$ can be understood as polynomial $*$-algebras generated by elements $A$, $B$ and $C_\pm$ that satisfy relations in Theorem~\ref{thm.structure}. By the Diamond Lemma such algebras have linear bases $A^rC_\pm^{s}$, $A^rC_\pm^{*s}$, $B^rC_\pm^{s}$, $B^rC_\pm^{*s}$. The assignments described in  Theorem~\ref{thm.structure} determine surjective $*$-algebra maps onto $\cO(S^2_{pq }(k,l^\pm))$ viewed as subalgebras of $\cO(S_{p q \theta}^3)$. These send basis elements into linearly independent elements hence are also injective and thus establish isomorpisms between $\cO(S^2_{pq }(k,l^\pm))$ and algebras generated by elements $A$, $B$ and $C_\pm$ that satisfy relations in Theorem~\ref{thm.structure}.
\end{remark}

\subsection{Representations of $\cO(S^2_{pq}(k,l^{\pm})$}
Bounded irreducible $*$-representations of coordinate algebras $\cO(S^2_{pq}(k,l^{\pm})$ are derived and classified by standard methods applicable to all algebras of this kind; see in particular the proof of \cite[Theorem~2.1]{HajMat:loc} or, for example, the proof of \cite[Proposition~2.2]{BrzFai:tea}.  
\begin{proposition}\label{prop.rep}
Up to unitary equivalence, the following is the list of all  bounded irreducible $*$-representations of $\cO(S^2_{pq}(k,l^{\pm})$. For all $m\in \NN$, let  $\cV_m \cong l^2(\NN)$ be a separable Hilbert space with orthonormal basis $e_n^m$ for $n \in \NN$. For $s=0,1,...,|l|-1$, $t=0,...,k-1$,  the representations  $\pi^{1}_s:\cO(S^2_{pq}(k,l^+)) \rightarrow \Endd(\cV_s)$, $\pi^{2}_t:\cO(S^2_{pq}(k,l^+)) \rightarrow \Endd(\cV_t)$ and $\pi^{-1}_s:\cO(S^2_{pq}(k,l^-)) \rightarrow \Endd(\cV_s)$ and $\pi^{-2}_t:\cO(S^2_{pq}(k,l^-)) \rightarrow \Endd(\cV_t)$ are given by 
\begin{subequations}
\begin{equation}
\pi_{s}^{\pm 1}(A)e_n^s=p^{n|l|+s}e_n^s, \qquad \pi_{s}^{\pm 1}(B)e_n^s=0, \qquad \pi_{s}^{1}(C_+)e_n^s=\prod_{i=1}^{l}(1-p^{i+nl+s})^{1/2} e_{n+1}^s,
\end{equation}
\begin{equation}
\pi_s^{-1}(C_-)e_n^s=\prod_{i=1}^{\lvert l \rvert}(1-p^{i+(n-1)\lvert l \rvert+s})^{1/2} e_{n-1}^s, 
\end{equation}
\begin{equation}
\pi_t^{\pm 2}(A)e_n^t=0, \quad \pi_t^{\pm 2}(B)e_n^t=q^{nk+t}e_n^t, \quad \pi_t^{\pm 2}(C_\pm)e_n^t=\prod_{j=1}^{k}(1-q^{j+(n-1)k+t})^{1/2} e_{n-1}^t.
\end{equation}
\end{subequations}
Furthermore, there are one-dimensional representations in each case given by $A,B \mapsto 0, C_{\pm} \mapsto \lambda$ where $\lambda \in \CC$ such that $\lvert \lambda \rvert=1$, which we denote by $\pi^\pm_{\lambda}$.
\end{proposition}

\subsection{Quantum Heegaard spaces as generalised Weyl algebras}

Generalised Weyl algebras are defined in \cite{Bav:gweyl} as follows.

\begin{definition}
Let $\cD$ be a ring, $\sigma=(\sigma_1,...,\sigma_n)$ a set of commuting automorphisms of $\cD$ and $\tilde{a}=(\widetilde{a_1},...,\widetilde{a_n})$ a set of (non-zero) elements of the centre $Z(\cD)$ of $\cD$ such that $\sigma_i(\widetilde{a_j})=\widetilde{a_j}$ for all $i \neq j$.
The  associated {\em generalised Weyl algebra $\cD(\sigma,\tilde{a})$ of degree $n$} is a ring generated by $\cD$ and the $2n$ indeterminates $X_1^+,...,X_n^+,X_1^-,...,X_n^-$ subject to the following relations, for all $\alpha \in \cD$:
\begin{equation}
X_i^-X_i^+ =\widetilde{a_i}, \qquad X_i^+X_i^-=\sigma_i(\widetilde{a_i}), \qquad X^{\pm}_i \alpha = \sigma_i^{\pm 1}(\alpha)X_i^{\pm}, \label{GWA rel1}
\end{equation}
\begin{equation}
[X_i^-,X_j^-]=[X_i^+,X_j^+]=[X_i^+,X_j^-]=0, \quad \forall i \neq j, \label{GWA rel2}
\end{equation}
where $[x,y]=xy-yx$.
\end{definition}

We call $\tilde{a}$ the {\em defining elements} and $\sigma$ the {\em defining automorphisms} of $\cD(\sigma,\tilde{a})$. Note that in the degree one  case the relations $(\ref{GWA rel2})$ are null.

\begin{proposition}
The algebras of coordinate functions $\cO(S_{pq}^2(k,l^{\pm}))$ are degree one generalised Weyl algebras.
\end{proposition} 

\begin{proof}
Set  $\cD=\CC[A,B]/\<AB,BA\>$. In the positive case $\cO(S_{pq}^2(k,l^+))$,  $X^+=C_+^*$ and $X^-=C_+$,  the automorphism $\sigma_+$ of $\cD$  and the  defining element  $\tilde{a}$ are 
$$
\sigma_+(A)=p^{l}A, \qquad \sigma_+(B)=q^{-k}B, \qquad \tilde{a}=\prod_{i=1}^l\prod_{j=1}^{k}(1-p^{i-l}A)(1-q^{j}B).
$$

In the negative case $\cO(S_{pq}^2(k,l^-))$,  $X^+=C_-$ and $X^-=C_-^*$,  the automorphism $\sigma_-$ of $\cD$ and the  defining element  $\tilde{a}$ are 
$$
\sigma_-(A)=p^{-l}A, \qquad \sigma_-(B)=q^{k}B, \qquad \tilde{a}=\prod_{i=1}^{|l|}\prod_{j=1}^{k}(1-p^{i+l}A)(1-q^{j-k}B).
$$
That these satisfy all the axioms of a degree one generalised Weyl algebra can be checked by routine calculations. 
\end{proof}

One of the key theorems associated to generalised Weyl algebras \cite[Theorem~1.6]{Bav:gweyl} provides an insight to the global dimension of such algebras. 
For example, one can prove that the coordinate algebra $\cO(\WP_q(k,l))$ of the quantum weighted projective line defined in \cite{BrzFai:tea} is a degree one generalised Weyl algebra, and then use \cite[Theorem~1.6]{Bav:gweyl} to conclude that the global dimension of $\cO(\WP_q(k,l))$ is equal to 2 if $k=1$ and is infinite otherwise.\footnote{We are grateful to Ulrich Kr\"ahmer for pointing this out to us.} This can be used as an indication that, for the quantum teardrop case $k=1$, the classical singularity has been removed (although it is not clear yet, whether   $\cO(\WP_q(1,l))$ is a Calabi-Yau algebra). Unfortunately, the hypothesis of \cite[Theorem~1.6]{Bav:gweyl}  fails in the quantum Heegaard case since the basic ring $\cD$ contains zero divisors. On the other hand, one should not expect the global dimension of $\cO(S^2_{pq}(k,l^{\pm}))$ to be finite:  On the classical level the relation $AB=0$ implies that there is a singularity at the origin, which persists in the quantum case.

\section{Quantum weighted Heegaard spheres and quantum principal bundles}  \label{sec.princ} 
Recall from \cite{BrzHaj:Che} that, given a Hopf algebra $\cH$, a right $\cH$-comodule algebra $\cA$ (with coaction $\rho$) is said to be a {\em Hopf-Galois  extension} of its coinvariant subalgebra $\Bb$ if the canonical map
$$
\can : \cA\ot_\cB \cA\to\cA\ot \cH, \qquad a\ot a'\mapsto a\rho(a'),
$$
is bijective. If, in addition, the mutliplication map $\cB\ot \cA\to\cA$ splits as a left $\cB$-module and right $\cH$-comodule map, then $\cA$ is called a {\em principal comodule algebra} (and the coaction $\rho$ is said to be {\em principal}). 
Here we aim to construct $\cO(U(1))$-principal comodule algebras with coinvariant subalgebras $\cO(S^2_{pq}(k,l^{\pm}))$. These can be understood geometrically as coordinate algebras of principal circle bundles over the quantum weighted Heegaard sphere.
We follow the general strategy  (employed previously in \cite{BrzFai:tea} and \cite{BrzFai:rps}) of defining a cyclic group algebra coaction on $\cO(S^3_{pq \theta})$ where the space of coinvariant elements with respect to this coaction forms the total space. In Section~\ref{sec.orb} we give a categorical interpretation of this strategy.

\subsection{Circle bundles over $S^2_{pq}(k,l^{+})$}
Since the fixed point algebra of $(\cO(S_{p q \theta}^3), \phi_{k,l})$, for positive $l$, is generated by $C_+=a^lb^{*k}$, $A$ and  $B$, we need to define a comodule structure on $\cO(S_{p q \theta}^3)$ over the cyclic group algebra $\cO(\ZZ_m)$ that keeps these generators in the invariant part. In terms of  the $\ZZ_m$-grading this means
$$
\deg(a^l b^{*k})=l \ \deg(a)-k \ \deg(b) = 0 \ \text{mod} \ m.
$$
This equation is satisfied by setting $\deg(a)=k$, $\deg(b)=l$ and $m=kl$. The grading is equivalent to a coaction $\Lambda_{k,l}: \cO(S_{pq \theta}^3) \rightarrow \cO(S_{pq \theta}^3) \otimes \cO(\ZZ_{kl})$ given by $a \mapsto a \otimes u^k, b \mapsto b \otimes u^{l}$, which is extended in the usual way to make $\Lambda_{k,l}$ an algebra map and hence $(\cO(S^3_{pq \theta}), \Lambda_{k,l})$ a comodule algebra. The fixed point subalgebra is generated by $x=a^l, y=b^{k}, z=A,w=B$. This can be seen by taking a basis element and applying $\Lambda_{k,l}$:  $A^{\lambda} a^{\mu} b^{\nu}$ is coinvariant if and only if 
$k \mu+l \nu = 0 \ \text{mod} \ kl,$
hence
$\mu=l \phi$ and $\nu=k \delta$, for some $\phi, \delta \in \ZZ$. 
This means that $A^{\lambda} a^{\mu} b^{\nu} = A^{\lambda} (a^l)^{\phi} (b^k)^{\delta}$ so $a^l$, $b^k$ and $A$ are generators. By considering the other basis element the final generator $B$ is identified. The resulting algebra  $\mathcal{A}_{k,l}=\cO(S^3_{pq \theta})^{co \cO(\ZZ_{kl})}$ is the quotient of $\CC[w,x,y,z]$ by the relations
\begin{subequations}\label{arel}
\begin{equation}
xy=e^{2 \pi i \theta kl}yx, \qquad x^*y=e^{-2 \pi i \theta kl} yx^*, 
\end{equation}
\begin{equation}\label{arel.2}
xx^*=\prod_{i=1}^{l}(1-p^{i-l}z), \ \ x^*x=\prod_{i=1}^l (1-p^iz), \ \ yy^*=\prod_{i=1}^{k} (1-q^{i-k}w),  \ \ y^*y=\prod_{i=1}^k(1-q^iw),
\end{equation}
\begin{equation}
z^*=z, \ \ w^*=w, \ \ wz=zw=0, \ \ xw=wx, \ \ yz=zy, \ \ yw=q^{-k}wy, \ \ xz=p^{-l}zx.
\end{equation}
\end{subequations}

The circle group algebra coacts on $\mathcal{A}_{k,l}$ by 
$$
\rho_{k,l}: \mathcal{A}_{k,l} \rightarrow \mathcal{A}_{k,l} \otimes \cO(U(1)), \qquad w \mapsto w \otimes 1, \ x \mapsto x \otimes u, \ y \mapsto y \otimes u, \ z \mapsto z \otimes 1,
$$
making  $(\mathcal{A}_{k,l}, \rho_{k,l})$ a $\cO(U(1))$-comodule algebra. The fixed points of this comodule algebra are generated by $\alpha=w$, $\beta=z$, $\gamma=xy^*$, and thus are isomorphic to $\cO(S^2_{pq}(k,l^+))$ via
the map 
$\alpha \mapsto B, \beta \mapsto A, \gamma \mapsto C_+$.

\begin{theorem} $(\mathcal{A}_{k,l}, \rho_{k,l})$ is a principal $\cO(U(1))$-comodule algebra over $\cO(S^2_{pq}(k,l^+))$.
\end{theorem}

\begin{proof}
To prove principality we construct a {\em strong connection} for $\cA_{k,l}$, that is a map
$
\omega:\cO(U(1))\longrightarrow \cA_{k,l}\otimes \cA_{k,l},
$
normalised as $\omega(1)=1\otimes 1$, and  such that
\begin{subequations}
\label{strong}
\begin{gather}
\mu\circ \omega = \eta \circ \varepsilon, \label{strong2}\\
 (\omega\otimes\id)\circ\Delta  = (\id\otimes \rho_{k,l})\circ \omega , \label{strong3}\\
(S\otimes \omega)\circ\Delta = (\tw\otimes \id)\circ (\rho_{k,l}\otimes \id)\circ \omega .  \label{strong4}
\end{gather}
\end{subequations}
Here $\mu: \cA_{k,l}\ot \cA_{k,l}\to \cA_{k,l}$ denotes the multiplication, $\eta: \CC\to \cA_{k,l}$ is the unit map,  and $\tw : \cA_{k,l}\ot \cO(U(1)) \to \cO(U(1))\ot \cA_{k,l}$ is the flip. A strong connection for $\mathcal{A}_{k,l}$ is defined by setting $\omega(1)=1 \otimes 1$ and then recursively  for $n \in \NN$,
\begin{subequations} \label{strong.+}
\begin{equation} \label{strong.+1}
\omega(u^n)=x^*\omega(u^{n-1})x+f(z)y^* \omega(u^{n-1})y,
\end{equation}
\begin{equation} \label{strong.+2}
\omega(u^{-n})=x \omega(u^{-n+1})x^*+f(p^{-l}z)y \omega(u^{-n+1})y^*,
\end{equation}
\end{subequations}
where $f(z) = 1-\prod_{i=1}^l (1-p^i z)$. That thus defined $\omega$ satisfies \eqref{strong} is proven by induction.

\underline{Condition \eqref{strong2}}: Note that, since $y^*y$ is a polynomial in $w$ with constant term 1, $zw=0$  and $f(z)$ has the zero constant term, $f(z)y^*y =f(z)$. Hence, in the case $n=1$,
$$
(\mu \circ \omega)(u)=x^*x+f(z)y^*y = x^*x + 1-\prod_{i=1}^l (1-p^i z) = 1, 
$$
by \eqref{arel.2}.
Now assume that $(\mu \circ \omega)(u^n)=1$ and consider
\begin{equation*}
\begin{aligned}
(\mu \circ \omega)(u^{n+1})&=x^*(\mu \circ \omega(u^{n}))x+f(z)y^* (\mu \circ \omega(u^{n}))y 
=x^*x+ f(z)y^*y =1,
\end{aligned}
\end{equation*}
where the second equality is the inductive assumption. The proof for negative powers of $u$ is essentially the same, so the details are omitted.

\underline{Condition \eqref{strong3}}: 
Consider $\omega$ for positive powers of $u$; putting $n=1$ gives,
\begin{equation*}
\begin{aligned}
((\id \otimes \rho_{k,l}) \circ \omega)(u)&=(\id \otimes \rho_{k,l}) (x^* \otimes x+ f(z)y^* \otimes y) \\
&=x^* \otimes x \otimes u+ f(z)y^* \otimes y \otimes u 
=\omega(u) \otimes u.
\end{aligned}
\end{equation*}
This is the basis for the inductive proof. Now assume that
$$
((\id \otimes \rho_{k,l}) \circ \omega)(u^n) = ((\omega \otimes \id) \circ \Delta)(u^n)= \omega(u^n)\ot u^n, 
$$
and consider
\begin{equation*}
\begin{aligned}
((\id \otimes \rho_{k,l}) \circ \omega)(u^{n+1})&=(\id \otimes \rho_{k,l})(x^*\omega(u^{n})x+f(z)y^* \omega(u^{n})y) \\
&=x^*\omega(u^{n})x \otimes u^{n+1}+f(z)y^*\omega(u^{n})y \otimes u^{n+1} 
=\omega(u^{n+1}) \otimes u^{n+1},
\end{aligned}
\end{equation*}
where the second equality follows from the induction hypothesis and $\rho_{k,l}$ being an algebra map. The proof for $\omega$ taking negative powers of $u$ follows a similar argument. 

\underline{Condition \eqref{strong4}}: As in the previous conditions, the proofs for positive and negative powers of $u$ are similar, hence only the positive case is displayed. Note that here we need to show that
\begin{equation}\label{asser}
((\tw \otimes \id) \circ (\rho_{k,l} \otimes id) \circ \omega )(u^n)=((S \otimes \omega ) \circ \Delta )(u^n) = u^{-n}\ot \omega(u^n),
\end{equation}
for all $n$. The case $n=1$ follows by the same argument as in the preceding proof. 
Assume that equation \eqref{asser} is true for an $n\in \NN$, 
and consider
\begin{equation*}
\begin{aligned}
((\tw &\otimes \id) \circ (\rho_{k,l} \otimes \id) \circ \omega)(u^{n+1})=((\tw \otimes \id) \circ (\rho_{k,l} \otimes \id))(x^*\omega(u^{N})x
 +f(z)y^*\omega(u^{n})y) \\
&=u^{-n-1} \otimes x^*\omega(u^{n})x+u^{-n-1} \otimes f(z)y^* \omega(u^{n})y  
=u^{-(n+1)} \otimes \omega(u^{n+1}) ,
\end{aligned}
\end{equation*}
where the second equality follows from the induction hypothesis and $\rho_{k,l}$ being an algebra map. This completes the proof of  \eqref{strong4} for positive powers of $u$.

Since $(\cA_{k,l},\rho_{k,l})$ is a comodule algebra admitting a strong connection it is principal; see \cite{DabGro:str}, \cite{BrzHaj:Che}.
\end{proof}

Recall that an $\cH$-comodule algebra $\cA$ is said to be {\em cleft} provided there is a convolution invertible $\cH$-comodule map $j: \cH\to \cA$. A cleft comodule algebra is automatically principal and corresponds to a principal bundle which only admits trivial associated vector bundles. A comodule algebra is said to be {\em trivial}, if there exists an $\cH$-comodule algebra map $j: \cH\to \cA$. A trivial comodule algebra is automatically cleft. 

\begin{proposition}
The principal $\cO(U(1))$-comodule algebra $\mathcal{A}_{k,l}$ is not cleft.
\end{proposition}

\begin{proof}
Since $\mathcal{A}_{k,l} \subseteq  \cO(S_{pq \theta}^3)$ and, by  \cite[Theorem~1.10]{HajRen:len}, the only invertible elements in $\cO(S_{pq \theta}^3)$ are multiples of $1$, the only invertible elements in $\mathcal{A}_{k,l}$ are also the multiples of $1$. The convolution invertible map $j: \cO(U(1)) \rightarrow \mathcal{A}_{k,l}$ must take the form $j(u)= \alpha 1$ for some $\alpha \in \CC^*$. However this violates the right $\cO(U(1))$-colinearity of $j$.
\end{proof}

\subsection{Circle bundles over $S^2_{pq}(k,l^{-})$}

Following the process as in the case of positive $l$, a suitable coaction  $\Phi_{k,l}: \cO(S_{pq \theta}^3) \rightarrow \cO(S_{pq \theta}^3) \otimes \cO(\ZZ_{k \lvert l \rvert})$ for a negative $l$  is arrived at as 
 given by $a \mapsto a \otimes u^k, b \mapsto b \otimes u^{l}$.
The fixed point subalgebra is generated by $x=a^{\lvert l \rvert}, y=b^{k}, z=A,w=B$. This coincides with the algebra $\mathcal{A}_{k,-l}$, hence positive and negative values of $l$ give rise to the same total space of the quantum principal bundle. The principal coaction of $\cO(U(1))$ on $\mathcal{A}_{k,-l}$ that fixes $\cO(S_{pq}^2(k,l^-))$ is defined by $x\mapsto x\ot u$, $y\mapsto y\ot u^*$, $z\mapsto z\ot 1$ and $w\mapsto w\ot 1$.

\subsection{Piecewise triviality}\label{sec.loc}
$\mathcal{A}_{1,1}$ is a prime example of a locally trivial quantum principal bundle \cite{BudKon:pri}, \cite{CalMat:con}, \cite{HajMat:loc},  \cite{Zie:loc} or a piecewise trivial comodule algebra \cite[Definition~3.8]{HajKra:pie}. In a similar way $\mathcal{A}_{k,|l|}$ is a piecewise trivial principal comodule algebra over $\cO(S^2_{pq}(k,l^\pm))$. Consider ideals $\langle w\rangle$, $\langle z\rangle$ of $\mathcal{A}_{k,|l|}$ generated by $w$ and $z$ respectively. Clearly $\langle w\rangle \cap \langle z\rangle = 0$, hence these ideals form a {\em complete covering} of   $\mathcal{A}_{k,|l|}$ \cite{CalMat:cov}. Each one of them is a subcomodule of $\mathcal{A}_{k,|l|}$ under the $\cO(U(1))$-coactions, hence the quotient algebras $\mathcal{A}_{k,|l|}/ \langle w\rangle$ and $\mathcal{A}_{k,|l|}/ \langle z\rangle$ are $\cO(U(1))$-comodule algebras. In view of relations \eqref{arel}, the class $[y]$ is a unitary element of $\mathcal{A}_{k,|l|}/ \langle w\rangle$, while the class $[x]$ is a unitary element of $\mathcal{A}_{k,|l|}/ \langle z\rangle$. Hence one can define right $\cO(U(1))$-colinear $*$-algebra maps 
$$
j_w: \cO(U(1)) \to \mathcal{A}_{k,|l|}/ \langle w\rangle, \quad u \mapsto [y]^{\sign(l)}; \qquad j_z: \cO(U(1)) \to \mathcal{A}_{k,|l|}/ \langle w\rangle, \quad u \mapsto [x].
$$
Consequently, $\mathcal{A}_{k,|l|}/ \langle w\rangle$ and $\mathcal{A}_{k,|l|}/ \langle z\rangle$ are trivial principal comodule algebras.

The fixed points of  $\cO(U(1))$-coactions on $\mathcal{A}_{k,|l|}/ \langle w\rangle$ and $\mathcal{A}_{k,|l|}/ \langle z\rangle$ come out as quotients
$\cO(S^2_{pq}(k,l^\pm))/\langle B\rangle$ and $\cO(S^2_{pq}(k,l^\pm))/\langle A\rangle$, respectively. These are isomorphic to coordinate algebras of quantum real projective planes $\cO(\RR\PP^2_{\sqrt{p}}(k;+))$ and $\cO(\RR\PP^2_{\sqrt{q}}(|l|;+))$, respectively, introduced in \cite{Brz:sei}. Again, $\langle A\rangle\cap \langle B\rangle=0$, so the ideals $\langle A\rangle$, $\langle B\rangle$ constitute a complete covering of $\cO(S^2_{pq}(k,l^\pm))$. Therefore, $\mathcal{A}_{k,|l|}$ are piece trivial principal comodule algebras  over $\cO(S^2_{pq}(k,l^\pm))$ in the sense of  \cite[Definition~3.8]{HajKra:pie}.

\subsection{Orbifolding}\label{sec.orb} 

The $\ZZ$-grading of the algebra $\cO(S_{p q \theta}^3)$ that leads to the definition of $\cO(S_{pq}^2(k,l^\pm))$ gives different weights to generators $a$ and $b$. There is a way of introducing new, uniform grading, which, following \cite{CalTu:cur} we term {\em orbifolding}. The algebras $\cA_{k,l}$ introduced in the previous sections gain in this way a natural, categorical interpretation. 

Consider $\cO(S_{p q \theta}^3)$ as a $\ZZ$-graded algebra with grading compatible with $*$ and  determined from $\deg(a) =k$, $\deg(b) = l$ ($k$ positive). The cyclic group $\ZZ_{k|l|}$ acts on $\cO(S_{p q \theta}^3)$ by
$$
j\la h = \zeta^{j\deg(h)}h,
$$
where $h$ is a homogeneous element of  $\cO(S_{p q \theta}^3)$  and $\zeta$ is a primitive $k|l|$-th root of unity. We can form a crossed product $\cO(S_{p q \theta}^3)\cross \ZZ_{k|l|}$ and view it as a category with objects $\chi_i$, $i=0, \ldots, k|l|$, the characters of $\ZZ_{k|l|}$, and morphism sets
$$
\rhom{} {\chi_i} {\chi_j} = \{ h\in\cO(S_{p q \theta}^3)\; |\; \deg(h) = (j-i) \mod k|l|\}.
$$
The composition is provided by the multiplication in $\cO(S_{p q \theta}^3)$. Each of the hom-sets $\rhom{} {\chi_i} {\chi_j} $ is a $\ZZ$-graded space with grading
$$
\widehat{h} := \frac{\deg(h) -j +i}{k|l|}.
$$
This grading is compatible with composition, hence each of the $\rend {} {\chi_i}$ is a $\ZZ$-graded algebra. One easily finds that $\rend {} {\chi_i}$ is a $*$-subalgebra 
of $\cO(S_{p q \theta}^3)$ generated by $A,B, a^{|l|}, b^k$, i.e.\ it is simply $\cA_{k,l}$ of the preceding sections. The hat-gradings come out as:
$$
\widehat{A} = \widehat{B} = 0, \qquad \widehat{a^{|l|}} = 1, \qquad \widehat{b^{k}} = \sign (l),
$$
i.e.\ they coincide with the coactions that make $\cA_{k,l}$ a principal comodule algebra over $\cO(S_{pq}^2(k,l^\pm))$.

\section{Fredholm modules and the Chern-Connes pairing for $\cO(S_{pq}^2(k,l^\pm))$}  \label{sec.Fred}

In this section first we associate even Fredholm modules to algebras $\cO(S_{pq\theta}^3)$ and use them to construct traces or cyclic cocycles on $\cO(S_{pq\theta}^3)$. The latter are  then used  to calculate the Chern number of a non-commutative line bundle associated to the quantum principal bundle $\mathcal{A}_{k,|l|}$ over the quantum weighted Heegaard spaces $\cO(S_{pq}^2(k,l^\pm))$.

First, recall from \cite[Chapter~IV]{Con:non} that an {\em even Fredholm module} over a $*$-algebra $\cA$ is a quadruple $(\cV,\pi, F,\gamma)$, where $\cV$ is a Hilbert space of a representation $\pi$ of $\cA$ and $F$ and $\gamma$ are operators on $\cV$ such that $F^* =F$, $F^2 =I$,  $\gamma^2 =I$, $\gamma F = -\gamma F$, and, for all $x\in \cA$, the commutator $[F,\pi(x)]$ is a compact operator. A Fredholm module is said to be {\em 1-summable} if $[F,\pi(x)]$ is a trace class operator for all $x\in \cA$. In case $\cV$ is a separable Hilbert space the last condition means
$$
\text{Tr} \ | [F,\pi(x)]|=\sum_{k \in \NN}\< ([F,\pi(x)]^* [F,\pi(x)])^{\frac{1}{2}}e_k,e_k \> < \infty, \qquad \mbox{for \ all} \ \ x \in \cA,
$$
where $\{e_k\}_{k \in \NN}$ is an orthonormal basis for $\cV$.

\subsection{Fredholm modules}
Representations for Fredholm modules are constructed from irreducible representations of $\cO(S_{pq}^2(k,l^\pm))$ listed in Proposition~\ref{prop.rep}. In the positive $l$ case we take a cue from \cite{HajMat:ind} and, for every $s=0,1,2,\ldots , l-1$, $t=0,1,2,\ldots , k-1$, consider the representation $\pi_{s,t}$ obtained as a direct sum of representations $\pi_s^1$ and $\pi_t^2$. The Hilbert space of this representation is denoted by $\cV_{s,t}$ and we  choose its orthonormal basis  $f^{s,t}_m$, $m\in \ZZ$ as follows.  For $m$ positive the $f^{s,t}_m$ correspond to the basis elements $e_m^s$ of the representation space of $\pi_s^1$, and for negative $m$, the $f^{s,t}_m$ correspond to the $e_{-m-1}^t$ of the Hilbert space of $\pi_t^2$.  In addition to $\pi_{s,t}$ we also consider the integral of one-dimensional representations,  $\pi_c  = \int_{\lambda\in S^1} \pi^+_\lambda \mathrm{d}\lambda$. The representation space of $\pi_c$ can be identified with $\cV_{s,t}$, so that
\begin{equation} \label{pic}
\pi_c(A)=\pi_c(B)=0, \ \ \pi_c(C_+^{\mu})f_m^{s,t}=f_{m+\mu}^{s,t}, \ \ \pi_c(C_+^{* \mu}) f_m^{s,t}= f_{m-\mu}^{s,t},
\end{equation}
for all $m\in \ZZ$, $\mu\in \NN$.

\begin{proposition} \label{chern.char.+}
For all $s=0,1,2,\ldots , |l|-1$, $t=0,1,2,\ldots , k-1$, 
$(\cV_{s,t}\oplus\cV_{s,t} , \bar{\pi}_{s,t}:= \pi_{s,t} \oplus \pi_c ,F, \gamma)$,  
where
$$
 F = \begin{pmatrix} 0 & I \cr I & 0\end{pmatrix}, \qquad \gamma = \begin{pmatrix} I & 0 \cr 0 & -I\end{pmatrix},
$$
is a 1-summable Fredholm module over $\cO(S_{pq}^2(k,l^+))$, while $(\cV_{s}\oplus\cV_{t} , \bar{\pi}^-_{s,t}:= \pi^{-1}_{s} \oplus  \pi^{-2}_{t} ,F, \gamma)$ is a 1-summable Fredholm module over $\cO(S_{pq}^2(k,l^-))$.

The corresponding Chern characters are
\begin{equation}\label{tau}
 \tau^\pm_{s,t} (A^{\lambda} C_\pm^{\#\mu}) =
 \begin{cases} \frac{p^{\lambda s}}{1-p^{\lambda |l|}}   & \mbox{if $\mu=0$, $\lambda \neq 0$},\\
 0  & \mbox{otherwise},
 \end{cases}
\qquad
 \tau^\pm_{s,t} (B^{\lambda} C_\pm^{\#\mu}) =
 \begin{cases} \frac{q^{\lambda t}}{1-q^{\lambda k}}   & \mbox{if $\mu=0$, $\lambda \neq 0$},\\
 0  & \mbox{otherwise}.
 \end{cases}
\end{equation}
Here 
$$
C_\pm^{\#\mu}:=\begin{cases}
C_\pm^{\mu},  & \mu\geq 0, \\
C_\pm^{*|\mu|}, & \mu<0 .
\end{cases}$$
\end{proposition}

\begin{proof}
It is obvious that $F^* =F$, $F^2 = \gamma^2 =I$ and $F\gamma + \gamma F =0$. We first deal with the positive $l$ case. By a straightforward calculation, for all $x\in \cO(S_{pq}^2(k,l^+))$,
$$
[F,\pi(x)] = \begin{pmatrix} 0 & \pi_c(x) - \pi_{s,t}(x) \cr \pi_{s,t}(x) - \pi_c(x) & 0\end{pmatrix}.
$$
Due to the derivation property of a commutator and the fact that the set of trace class operators is a $*$-ideal in the algebra of bounded operators, suffices it to 
show that $\pi_{s,t}(x)-\pi_c(x)$ are trace class for $x=A,B,C_+$.
Using the formulae in Proposition~\ref{prop.rep} one easily finds, 
for all $\lambda, \mu \in \NN$,
\begin{subequations}
\begin{equation}
\pi_{s,t}(A^\lambda C_+^{\mu})f_m^{s,t}=
\begin{cases}
p^{\lambda((m+\mu)l+s)} \prod_{i=1}^{l \mu}(1-p^{i+s+ml})^{1/2} f^{s,t}_{m+\mu} & m\geq 0, \\
0 & m=-1,...,-\mu, \\
\delta_{\lambda, 0}\prod_{i=1}^{k\mu}(1-q^{i-(m+1+\mu)k+t})^{1/2} f^{s,t}_{m+\mu} & m < -\mu.
\end{cases}
\end{equation}
and
\begin{equation}
\pi_{s,t}(C_+^{* \mu}A^\lambda)f_m^{s,t}=
\begin{cases}
0 & m=0,1,... , \mu-1, \\
p^{\lambda(ml+s)}\prod_{i=1}^{l \mu}(1-p^{i+s+(m-\mu)l})^{1/2} f^{s,t}_{m-\mu} & m \geq \mu, \\
\delta_{\lambda, 0}\prod_{i=1}^{k\mu}(1-q^{i-(m+1)k+t})^{1/2} f^{s,t}_{m-\mu} & m < 0. \\
\end{cases}
\end{equation}
\end{subequations}
Set $X^+_{\lambda,\mu}:=\pi_{s,t}(A^\lambda C_+^{\#\mu})-\pi_c(A^\lambda C_+^{\#\mu})$. Using \eqref{pic} 
we find 
\begin{equation*}
({X^+}^*_{0,\mu}X^+_{0,\mu})^{\frac{1}{2}}f_m^{s,t}=
\begin{cases}
1-\prod_{i=1}^{l \mu}(1-p^{i+s+ml})^{1/2}f_m^{s,t} & \qquad m \geq 0, \\
0 & m=-\mu,...,-1 \\
1-\prod_{i=1}^{k \mu}(1-q^{i+t-(m+\mu)l})^{1/2}f_m^{s,t} & \qquad m< -\mu, 
\end{cases}
\end{equation*}
and, for $\lambda \neq 0$,
\begin{equation}
({X^+_{\lambda, \mu}}^* X^+_{\lambda, \mu})^{\frac{1}{2}} f_m^{s,t}=
\begin{cases}
p^{\lambda((m+\mu)l+s)} \prod_{i=1}^{l \mu}(1-p^{i+s+ml})^{1/2} f^{s,t}_{m} & m\geq \mu, \\
0 & m < \mu.
\end{cases}
\end{equation}

We need  to estimate the two infinite series for $X^+_{0,1}$ and $X^+_{1,0}$. For the first series, note that all factors in the product are less than one, hence, for a non-negative $m$
\begin{equation*} 
\begin{aligned}
\< ({X^+_{0, 1}}^* X^+_{0, 1})^{\frac{1}{2}}f^{s,t}_m,f^{s,t}_m \>& = 1-\prod_{i=1}^{l}(1-p^{i+s+ml})^{\frac{1}{2}} 
\leq 1-\prod_{i=1}^{l}(1-p^{i+s+ml}) \\ &\leq 1-(1-p^{l+s+ml})^l  = p^{l+s+ml} \sum_{i=1}^l(-1)^i\binom l i p^{(l+s+ml)i} \\
&\leq p^{l+s+ml} \sum_{i=1}^l\binom l i \leq 2^lp^{s+ml +l}.
\end{aligned}
\end{equation*}
 Hence the sum over $m\geq 0$ can be estimated by a convergent geometric series and thus it is finite. Similarly for $m<0$, using an analogous identity, we find that the sum is  finite, so $X^+_{0,1}$ is a trace-class operator.   

For the second series,
\begin{equation*}
\begin{aligned}
\sum_{m=0}^{\infty}\< ({X^+_{1, 0}}^* X^+_{1, 0})^{\frac{1}{2}}f^{s,t}_m,f^{s,t}_m \>&=\sum_{m=0}^{\infty} p^{ml+s} = \frac{p^{s}}{1-p^{ l}} . 
\end{aligned}
\end{equation*}
By defining 
$
\tilde{X}^+_{\lambda, \mu}=(\pi_{s,t}-\pi_c)(B^{\lambda} C_+^{\#\mu}),
$
in an analogous way it can be shown that $\tilde{X}^+_{1, 0}$ is  also a trace class operator; hence we have shown that $(\cV_{s,t}\oplus \cV_{s,t}, \bar{\pi}_{s,t}, F, \gamma)$ is a 1-summable Fredholm module over $\cO(S_{pq}^2(k,l^+))$.

For the negative case, set $X^-_{\lambda,\mu}:=\pi^{-1}_{s}(A^\lambda C_-^{\#\mu})-\pi^{-2}_ t(A^\lambda C_-^{\#\mu})$, $\tilde{X}^-_{\lambda,\mu}:=\pi^{-1}_{s}(B^\lambda C_-^{\#\mu})-\pi^{-2}_ t(B^\lambda C_-^{\#\mu})$, the only non-zero entries of the commutator of $F$ with $\bar{\pi}^-_{s,t}$ evaluated at the basis elements of $\cO(S_{pq}^2(k,l^-))$. Using the formulae in Proposition~\ref{prop.rep} one finds that $\tr |X^-_{0,1}|$ can be estimated by the difference of two geometric series, while both $\tr |X^-_{1,0}|$ and  $\tr |\tilde{X}^-_{1,0}|$ are geometric series, very much as in the positive $l$ case.  Thus $(\cV_{s}\oplus\cV_{t} , \bar{\pi}^-_{s,t}:= ,F, \gamma)$ is a 1-summable Fredholm module over $\cO(S_{pq}^2(k,l^-))$ as required.

 Finally, we can calculate the Chern characters on the basis elements of $\cO(S_{pq}^2(k,l^\pm))$ using $\tau^\pm_{s,t}(x) = \tr (\gamma \bar{\pi}_{s,t}^\pm(x))$. First, for $\lambda\neq 0$,
\begin{equation*}
\tau^\pm_{s,t}(A^{\lambda}C_\pm^{\#\mu})=\tr( \gamma \bar{\pi}_{s,t}^\pm(A^{\lambda}C_\pm^{\#\mu})) =\tr(X^\pm_{\lambda, \mu}) = \delta_{\mu,0}\sum_{m=0}^{\infty} p^{\lambda(ml+s)}=\delta_{\mu,0}\frac{p^{\lambda s}}{1-p^{\lambda |l|}},
\end{equation*}
where we noted that  if  $\mu \neq 0$, then  all diagonal entries of $X^\pm_{\lambda, \mu}$ are zero. 
Similarly, by considering the traces of $\tilde{X}^\pm_{\lambda,\mu}$ we find that
$\tau^\pm_{s,t}(B^{\lambda}C_+^{\#\mu})= \delta_{\mu,0}
\frac{q^{\lambda t }}{1-q^{\lambda k}},$   
for $\lambda \neq 0$. 
\end{proof}

\subsection{The Chern-Connes pairing}
Projective modules of sections of  line bundles associated to a principal $\cO(U(1))$-comodule algebra $\cA$ with coaction $\rho$  and coinvariant subalgebra $\cB$ are defined as
$$
\cL[n] := \{ x\in \mathcal{A} : \rho(x) = x\ot u^n\}, \qquad n\in \ZZ.
$$
In other words, $\cL[n]$ is the degree $n$ component of $\mathcal{A}$ when the latter is viewed as a $\ZZ$-graded algebra. An idempotent $E[n]$ for $\cL[n]$ is given in terms of a strong connection $\omega$,
 \begin{equation}\label{projectors}
E[n]_{ij} = \omega(u^n)\su 2_i\omega(u^n)\su 1_j \in \cB,
\end{equation}
where $\omega(u^n) = \sum_i  \omega(u^n)\su 1 _i \ot \omega(u^n)\su 2_i$; see \cite[Theorem~3.1]{BrzHaj:Che}. 
The traces of powers of each of the $E[n]$ make up a cycle in the cyclic complex of $\cB$. In particular $\tr E[n]$ can be paired with the Chern character associated to a Fredholm module over $\cB$ to give an integer, which identifies isomorphism classes of the $E[n]$.  

\begin{theorem}\label{prop.pair}
For all $s=0,1,\ldots, l-1$, $t=0,1,\ldots,k-1$, let $\tau^+_{s,t}$ be the cyclic cocycle on $\cO(S_{pq}^2(k,l^+))$ constructed in Proposition~\ref{chern.char.+}. Let $E[n]$ be the idempotent determined by  $\omega(u^n)$ in \eqref{strong.+}. Then  $\tau^+_{s,t}(\tr E[n]) =-n$. Consequently, for $n\neq 0$, the left $\cO(S_{pq}^2(k,l^+))$-modules $\cL[n]$ corresponding to $E[n]$ are not free.
\end{theorem}
\begin{proof} 
We prove the theorem for the positive values of $n$. The negative $n$ case is  proven in a similar way. Define $f (z) = 1 - \prod_{i=1}^l (1-p^iz)$ and note that, for all $s=0,1,\ldots, l-1$, $f(p^{s-l}) =1$. 

\begin{lemma}\label{lem.pair.1}
For positive $n$,
\begin{equation}\label{strong.+3}
\omega(u^n) = \sum_i \omega(u^{n-1})\su 1_ix^* \ot  x\omega(u^{n-1})\su 2_i + \sum_i \omega(u^{n-1})\su 1_i f(z) y^*\ot  y\omega(u^{n-1})\su 2_i
\end{equation}
\end{lemma}
\begin{proof}
This is proven by induction. For $n=1$, this is simply equation \eqref{strong.+1} with $n=1$. Assume that equation \eqref{strong.+3} is true for any $r\leq n$. Then, using \eqref{strong.+1} and the inductive assumption, we can compute
\begin{eqnarray*}
\omega(u^{n+1}) &=& x^*\omega(u^n)x +f(z) y^*\omega(u^n)y \\
&=& \sum_i x^*\omega(u^{n-1})\su 1_i\underbrace{x^* \ot  x}\omega(u^{n-1})\su 2_ix\\
&& + \sum_i x^*\omega(u^{n-1})\su 1_i \underline{f(z) y^*\ot  y}\omega(u^{n-1})\su 2_ix\\
&& +\sum_i f(z)y^*\omega(u^{n-1})\su 1_i\underbrace{x^* \ot  x}\omega(u^{n-1})\su 2_iy\\
&& + \sum_i f(z)y^*\omega(u^{n-1})\su 1_i \underline{f(z) y^*\ot  y}\omega(u^{n-1})\su 2_iy\\
&=&\sum_i \omega(u^{n})\su 1_ix^* \ot  x\omega(u^{n})\su 2_i + \sum_i \omega(u^{n})\su 1_i f(z) y^*\ot  y\omega(u^{n})\su 2_i,
\end{eqnarray*}
where we indicated grouping of terms over which the definition \eqref{strong.+1}  of the strong connection $\omega$ is applied. 
\end{proof} 
\begin{lemma}\label{lem.pair.2}
For all positive $n$, $\tr E[n] = g_n(z)$, where $g_n(z)$ is a polynomial in $z$ independent of $w$ such that
\begin{equation}\label{ind.g}
g_{n+1}(z) = \left(1-f\left(p^{-l}z\right)\right) g_n\left(p^{-l}z\right) + f(z)g_n(z).
\end{equation}
\end{lemma}
\begin{proof} By \eqref{strong.+3} and the definition of the idempotents $E[n]$,
\begin{eqnarray*}
\tr E[n] &=& \sum_i x\omega(u^{n-1})\su 2_i\omega(u^{n-1})\su 1_ix^*  + \sum_i   y\omega(u^{n-1})\su 2_i\omega(u^{n-1})\su 1_i f(z) y^*\\
&=& x\tr E[n-1]x^*  +   y\tr E[n-1]f(z) y^*.
\end{eqnarray*}
In particular, since $\tr E[0] =1$, 
\begin{equation}\label{g1}
\tr E[1] =xx^*  +   yf(z) y^* = xx^*  +    y y^*f(z) = 1-f\left(p^{-l}z\right) + f(z),
\end{equation}
where we used \eqref{arel} (expressed in terms of the polynomial $f$), in particular the fact that $z$ commutes with $y$, that $yy^*$ is a polynomial  in $w$ with constant term 1,  $wz=0$ and $f$ has the zero constant term. Therefore, $\tr E[1]$ is a polynomial in $z$ only (not in $w$), and it satisfies \eqref{ind.g} with $g_0 =1$. 

Assume, inductively, that $\tr E[n] = g_n(z)$. Then, again extracting the same information from \eqref{arel} as before and, additionally, using the commutation rule between $z$ and $x^*$ we obtain,
\begin{eqnarray*}
g_{n+1}(z) &=& xg_n(z)x^* + yg_n(z)y^*f(z) = xx^*g_n\left(p^{-l}z\right) + yg_n(z)y^*f(z)\\
& =&\left(1-f\left(p^{-l}z\right)\right) g_n\left(p^{-l}z\right) + f(z)g_n(z),
\end{eqnarray*}
as required.
\end{proof}

With Lemma~\ref{lem.pair.1} and Lemma~\ref{lem.pair.2} at hand we can complete the proof of Theorem~\ref{prop.pair}. Since $z=A$, the Chern number 
$$
\ch_n := \tau^+_{s,t}(\tr E[n]) = \tau^+_{s,t}(g_n(z)),
$$
is obtained by evaluating the powers of $z$ in polynomial expansion of $g_n(z)$ using formulae \eqref{tau}. We will proceed by induction on $n$, but first write
$$
f(z) = \sum_{m=1}^l c^l_m z^m, \qquad  g_n(z) = \sum_{r=0}^N d^n_r z^r.
$$
Note that $f(0) =0$, hence in view of \eqref{g1}, $g_0(0) =1$, and, consequently $g_n(0)=1$, by \eqref{ind.g}. Therefore, $d^n_0=1$. 

Apply \eqref{g1} and \eqref{tau}  to calculate
\begin{eqnarray*}
\ch_1 &=& \tau^+_{s,t}(g_1(z)) =\tau^+_{s,t}\left( 1-f\left(p^{-l}z\right) + f(z)\right)\\
&=& \tau^+_{s,t}\left(1 -\sum_{m=1}^l c^l_m(p^{-lm}-1)z^m\right)
= -\sum_{m=1}^l c^l_mp^{-lm+sm} 
= -f(p^{s-l})  = -1.
\end{eqnarray*}
The last equality follows from the observation that since $s=0,1,\ldots ,l-1$, one of the factors in the product must vanish. Next, assume that $\ch_n =-n$, that is
\begin{equation}\label{ind.ch}
 \sum_{r=1}^N d^n_r \frac{p^{sr}}{1-p^{lr}}  = -n.
\end{equation}
 Then, using \eqref{ind.g}
\begin{eqnarray*}
\ch_{n+1} &=& \sum_{r=0}^N d^n_r p^{-lr} \tau^+_{s,t}(z^{r})
 - \sum_{m=1}^l \sum_{r=0}^N c^l_md^n_r p^{-l(m+r)} \tau^+_{s,t}(z^{m+r}) + \sum_{m=1}^l \sum_{r=0}^N c^l_md^n_r  \tau^+_{s,t}(z^{m+r})\\
&=& \sum_{r=1}^N d^n_r \frac{p^{(s-l)r}}{1-p^{lr}} 
 - \sum_{m=1}^l \sum_{r=0}^N c^l_md^n_r \frac{p^{-l(m+r)}-1}{1-p^{l(m+r)}}p^{{(m+r)}s} \\
 &=& \sum_{r=1}^N d^n_r \frac{p^{(s-l)r}}{1-p^{lr}} 
 -f(p^{s-l})\sum_{r=0}^N d^n_r p^{(r-l)s} \\
 &=& \sum_{r=1}^N d^n_r \left(\frac{p^{(s-l)r}}{1-p^{lr}} -p^{(r-l)s} \right) -1 = \sum_{r=1}^N d^n_r \frac{p^{sr}}{1-p^{lr}} -1  = -n-1,
 \end{eqnarray*}
by inductive assumption \eqref{ind.ch}. This completes the proof of the theorem.
\end{proof}

\section{Continuous functions on the quantum weighted
Heegaard spheres}  \label{sec.cont}

The $C^*$-algebras $C(S^2_{pq}(k,l^\pm))$ of continuous functions on the quantum weighted
Heegaard spheres are defined as completions of the direct sum of representations  classified in Proposition~\ref{prop.rep}. 
In this section we identify these algebras as pullbacks of the Toeplitz algebra and  calculate their K-groups, closely following the approach of \cite{She:Poi} and\cite{HajMat:rea} (see also \cite{Brz:sei}).

We think about the Toeplitz algebra $\cT$ concretely as the $C^*$-algebra generated by an unilateral shift $U$ acting on a separable Hilbert space $\cV$ with orthonormal basis $e_n$ by  $Ue_n = e_{n+1}$ ($\cV$ could be any of the spaces $\cV_s$ in Proposition~\ref{prop.rep}). $\cT$ can be thought of as the algebra of continuous functions on the quantum unit disc, and the restriction of these functions to the boundary circle $S^1$ yields the {\em symbol map},
$$
\sigma : \cT \to C(S^1), \qquad U\mapsto u,
$$
where $u$ is the unitary generator of $C(S^1)$.  Let $\iota$ be an automorphism of $C(S^1)$ given by $u\mapsto u^*$. For all  $k\in \NN$ and $l\in \ZZ$, define
$$
\cT^{k,l} = 
\{ (x_1, \ldots, x_{k+l})\in \cT^{\oplus k+l}\; |\; \sigma(x_1) = \ldots = \sigma(x_l) =\iota\circ \sigma(x_{l+1}) = \ldots =  \iota\circ \sigma(x_{k+l}) \}
$$
for positive $l$ and 
$$
\cT^{k,l} = 
\{(x_1, \ldots, x_{k-l})\in \cT^{\oplus k-l}\; |\; \sigma(x_1) = \ldots = \sigma(x_{k-l}) \}, 
$$
for negative $l$. In this section we prove
\begin{theorem} 
For all $k\in \NN$, $l\in \ZZ$,
\begin{equation}\label{s.expl}
C(S^2_{pq}(k,l^\pm)) \cong \cT^{k,l}. 
\end{equation}
Consequently,
\begin{equation}\label{s.k}
K^1(C(S^2_{pq}(k,l^\pm))) =0, \qquad  K^0(C(S^2_{pq}(k,l^\pm))) =\ZZ^{k+ |l|}.
\end{equation}
\end{theorem}
\begin{proof}
To see that \eqref{s.k} follows from \eqref{s.expl}, we simply observe that the standard exact sequence
\begin{equation} \label{seq.123a}
\xymatrix{0 \ar[r] & \mathcal{K} \ar[r] & \mathcal{T} \ar[r]^-{\sigma } & C(S^1) \ar[r] & 0,}
\end{equation}
that characterises the Toeplitz algebra in terms of compact operators on $\cV$, yields the exact sequence
\begin{equation} \label{seq.123.pull}
\xymatrix{0 \ar[r] & \mathcal{K}^{\oplus k+|l|} \ar[r] & \mathcal{T}^{k,l} \ar[r] & C(S^1) \ar[r] & 0.}
\end{equation}
The sequence \eqref{seq.123.pull} gives rise to a six-term exact sequence of $K$-groups, which can be studied precisely as in \cite[Section~4.2]{Brz:sei} to aid the derivation of the $K$-groups as stated.

Let $J_A^\pm$, $J_B^\pm$ denote the closed $*$-ideals of $C(S^2_{pq}(k,l^\pm))$ obtained by  completing of ideals $\Jj_A^\pm$, $\Jj_B^\pm$ of $\cO(S^2_{pq}(k,l^\pm))$  generated by $A$ and $B$,
and let $\psi: C(S^2_{pq}(k,l^\pm))\to C(S^2_{pq}(k,l^\pm))/(J_A^\pm\oplus J_B^\pm)$ be the canonical surjection. The image of $\psi$ is generated by $\psi(C_\pm)$. In view of the relations \eqref{X rel 2} and \eqref{X rel 4}, $\psi(C_\pm)$ is a unitary operator, hence $C(S^2_{pq}(k,l^\pm))/(J_A^\pm\oplus J_B^\pm)\cong C(S^1)$. Applying $\pi^{\pm 1}_s$, $\pi^{\pm 2}_t$ to $\Jj_A^\pm \oplus \Jj_B^\pm$, one finds that  
the images contain only compact operators on the corresponding representation spaces $\cV_s$, $\cV_t$. On the other hand, these images contain all 
projections onto one-dimensional spaces generated by basis elements, 
and all step-by-one operators with non-zero weights, hence their completions contain all compact operators. In this way 
$J_A^\pm\oplus J_B^\pm$ can be identified with the direct sum $\cK^{k+|l|}$. 

Since the direct sum of all irreducible representations of a $C^*$-algebra is faithful, 
$$
\bigoplus_{s=0}^{l-1} \pi^{\pm 1}_s \oplus \bigoplus_{t=0}^{k-1} \pi^{\pm 2}_t \oplus \bigoplus_{\lambda\in S^1} \pi^{\pm}_\lambda,
$$ 
are faithful representations of $C(S^2_{pq}(k,l^\pm))$. In the same way as in \cite{She:Poi}, $\pi^\pm_\lambda$ factor through $\pi^{\pm 1}_s$, $\pi^{\pm 2}_t$, hence also 
$$
\pi^{\pm}:= \bigoplus_{s=0}^{l-1} \pi^{\pm 1}_s \oplus \bigoplus_{t=0}^{k-1} \pi^{\pm 2}_t 
$$
are faithful.  It is clear that the images of $\pi^{\pm}$
are contained in $\cT^{k+|l|}$. On the other hand, by inspecting formulae in Proposition~\ref{prop.rep} one easily finds that $\pi^{1}_s(C_+) - U$, $\pi^{-1}_s(C_-) - U^*$, $\pi^{\pm 2}_t(C_\pm) -U^*$  are step-by-one operators with coefficients tending to zero. Therefore, they are compact operators and thus are in the kernel of the symbol map. This implies that  the image of $\pi^{\pm}$ is contained in $\cT^{k,l}$. Summarising the above discussion we obtain commutative diagram with exact rows
$$
\xymatrix{0 \ar[r] & \cK^{\oplus k +|l|} \ar[r]\ar@{=}[d] &C(S^2_{pq}(k,l^\pm))\ar[r]\ar@^{(->}[d]_{\pi^\pm} & C(S^1)\ar@{=}[d]  \ar[r] & 0\\
0 \ar[r] &\cK^{\oplus k+|l|}\ar[r] & \Tt^{k,l}\ar[r] & C(S^1) \ar[r] & 0.}
$$
This implies the isomorphism \eqref{s.expl}.
\end{proof}

We note in passing that the $K^0$-groups of $C(S^2_{pq}(k,l^\pm))$ are dependent on or detect the weights. This is in contrast to the classical (commutative) situation in which, irrespective of weights, the $K^0$-group of the complex weighted projective  line is equal to $\ZZ^2$ \cite{AlA:com}.

Finally, the description of $C(S^2_{pq}(k,l^\pm))$ in terms of pullbacks $\cT^{k,l}$ of Toeplitz algebras gives one an opportunity to follow \cite{Wag:fib} and construct  Fredholm modules over and compute the index pairing for $C(S^2_{pq}(k,l^\pm))$ in a more straightforward way than in the algebraic setup of Section~\ref{sec.Fred}. In particular, the very definition of $\cT^{k,l}$ allows one to combine pairs of projections
$$
\pr_{i}: \cT^{k,l} \to \cT, \qquad (x_1, \ldots, x_{k+|l|})\mapsto x_i,
$$
into Fredholm modules $(\pr_i,\pr_j)$, with no restriction on $i$ and $j$ for a negative $l$, and with $i,j \leq l$ or $i,j >l$, for a positive $l$, since the $\pr_i - \pr_j$ have their images in compact operators. We believe that these topological aspects of the quantum weighted Heegaard spheres are worth indpendent study, which however goes beyond the scope of the present Letter.

\section*{Acknowledgements} 
We would like to express our gratitude to the referees, whose valuable, insightful and detailed comments led to improvements in both the contents and presentation of this paper.

\end{document}